\documentclass[11pt]{article}

\usepackage{amsmath,amssymb}
\usepackage{amsthm,bm,framed,url}
\usepackage{graphicx}
\usepackage[monochrome]{color}
\usepackage{psfrag}
\usepackage{epstopdf}
\usepackage{wrapfig}
\usepackage{enumitem}
\usepackage{cite}
\usepackage{hyperref}
\usepackage{fullpage}

\usepackage{mathrsfs}

\newtheorem{proposition}{Proposition}
\newtheorem{lemma}{Lemma}

\DeclareMathOperator*{\minimize}{minimize}
\DeclareMathOperator*{\maximize}{maximize}
\newcommand{\st}{\text{subject to}}
\newcommand{\half}{\frac{1}{2}}

\newcommand{\real}[1]{\Re\left\{ #1 \right\}}
\newcommand{\imag}[1]{\Im\left\{ #1 \right\}}
\newcommand{\diag}[1]{\textsf{diag}\left\{ #1 \right\}}
\newcommand{\tr}[1]{{\rm tr}\left\{ #1 \right\}}

\newcommand{\R}{\mathbb{R}}
\newcommand{\C}{\mathbb{C}}
\renewcommand{\H}{\mathbb{H}}
\newcommand{\CN}{\mathcal{CN}}

\newcommand{\s}{\bm{s}}
\renewcommand{\b}{\bm{b}}
\newcommand{\x}{\bm{x}}
\newcommand{\f}{\bm{f}}

\newcommand{\smin}{\s_\text{min}}
\newcommand{\smax}{\s_\text{max}}

\newcommand{\y}{\bm{y}}
\newcommand{\z}{\bm{z}}
\renewcommand{\r}{\bm{r}}

\newcommand{\eye}{\bm{I}}
\newcommand{\Itilde}{\tilde{\eye}}
\newcommand{\T}{\bm{T}}
\newcommand{\X}{\bm{X}}
\newcommand{\F}{\bm{F}}
\newcommand{\A}{\bm{A}}
\newcommand{\B}{\bm{B}}

\newcommand{\eldar}[1]{{\color{blue}#1}}
\providecommand{\keywords}[1]{\textbf{Keywords:} #1}

\title{\bf Phase Retrieval from 1D Fourier Measurements: Convexity, Uniqueness, and Algorithms\ignorespaces
\thanks{Conference version of part of this work appears in ICASSP 2016 \cite{HuaEldSid:ICASSP2016}.}}

\author{
Kejun~Huang,
Yonina~C.~Eldar,
and
Nicholas~D.~Sidiropoulos
}

\begin{document}
\maketitle

\begin{abstract}
This paper considers phase retrieval from the magnitude of 1D over-sampled Fourier measurements, a classical problem that has challenged researchers in various fields of science and engineering. {\color{blue}We show} that an optimal vector in a least-squares sense can be found by solving a convex problem, thus establishing a hidden convexity in Fourier phase retrieval. {\color{blue}We then show} that the standard semidefinite relaxation approach yields the optimal cost function value (albeit not necessarily an optimal solution). A method is then derived to retrieve an optimal minimum phase solution in polynomial time.
Using these results, a new measuring technique is proposed which guarantees uniqueness of the solution, along with an efficient algorithm that can solve large-scale Fourier phase retrieval problems with uniqueness and optimality guarantees.
\end{abstract}
\keywords{
phase retrieval, over-sampled Fourier measurements, minimum phase, auto-correlation retrieval, semi-definite programming, alternating direction method of multipliers, holography.
}

\section{Introduction}
Phase retrieval seeks to recover a signal from the magnitudes of linear measurements \cite{PR-book}.
This problem arises in various fields, including crystallography \cite{crystallography}, microscopy, and optical imaging \cite{optics}, due to the limitations of the detectors used in these applications. Different types of measurement systems have been proposed, e.g., over-sampled Fourier measurements, short-time Fourier measurements, and random Gaussian, to name just a few (see \cite{shechtman2015phase,JEB16} for contemporary reviews), but over-sampled Fourier measurements are most common in practice. Two fundamental questions in phase retrieval are: i) is the signal uniquely determined by the noiseless magnitude measurements (up to inherent ambiguities such as global phase); and ii) is there an efficient algorithm that can provably compute an optimal estimate of the signal according to a suitable criterion?

This paper considers the 1D phase retrieval problem from over-sampled Fourier measurements. It is well known that there is no uniqueness in 1D Fourier phase retrieval,
i.e. there are multiple 1D signals with
the same Fourier magnitude. This is true even when we ignore trivial ambiguities which include a global phase shift, conjugate inversion and spatial shift.
Nonuniqueness also holds when the support of the signal is bounded
within a known range \cite{H64}. In addition, since the phase retrieval problem is nonconvex, there are no known algorithms that provably minimize the least-squares error in recovering the underlying signal \cite{convex}.

Existing methods that attempt to resolve the identifiability issue include introducing sparsity assumptions on the input signal~\cite{lu2011sparse,ohlsson2014conditions,Kishore:12,eldar_phase_2012}, taking multiple masked Fourier measurements~\cite{candes2014phase}, or using the short-time Fourier transform~\cite{jaganathan2015recovering,eldar2015sparse}.
Algorithmically, the most popular techniques for phase retrieval are based on alternating projections \cite{gerchberg1972practical,fienup,convex}, pioneered by Gerchberg and Saxton \cite{gerchberg1972practical}
and extended by Fienup \cite{fienup}. More recently, phase retrieval has been treated using semidefinite programming (SDP) and low-rank matrix recovery ideas \cite{candes2013phase,SESE11}.
Several greedy approaches to phase retrieval have also been introduced such as GESPAR \cite{SBE14,BE12}.
Semidefinite relaxation for phase retrieval, referred to as PhaseLift \cite{candes2013phase}, is known to recover the true underlying signal when the measurement vectors are Gaussian. However, in the case of Fourier measurements, as we are considering here, there are no known optimality results for this approach.

Despite the apparent difficulty in 1D Fourier phase retrieval, we establish in this paper that under certain conditions this problem can be solved optimally in polynomial time.
In particular, we first show that the least-squares formulation of this problem can always be optimally solved using semidefinite relaxation. Namely, methods such as PhaseLift do in fact optimize the nonconvex cost in the case of Fourier measurements. By slightly modifying the basic PhaseLift program we propose a new semidefinite relaxation whose solution is proven to be rank-one and to minimize the error. However, due to the nonuniqueness in the problem, the solution is not in general equal to the true underlying vector. To resolve this ambiguity, we propose a semidefinite relaxation that will always return the minimum phase solution which optimizes the least-squares cost. 

Based on this result, we suggest a new approach to measure 1D signals \eldar{which transforms} an arbitrary signal into a minimum phase counterpart by simply adding an impulse, so that identifiability is restored. We then recover the input using the proposed SDP. This measurement strategy resembles classic holography used commonly in optics \cite{gabor1948new}.
Finally, we propose an efficient iterative algorithm to recover the underlying signal without the need to resort to lifting techniques as in the SDP approach. Our method, \eldar{referred to as auto-correlation retrieval -- Kolmogorov factorization (CoRK),} 
%is based on the alternating direction method of multipliers (ADMM)~\cite{boyd2011distributed} and 
is able to solve very large scale problems efficiently.
Comparing to existing methods that provide identifiability, the proposed approach is easy to implement (both conceptually and practically), requires a minimal number of measurements, and can always be solved to global optimality using efficient computational methods.

The rest of the paper is organized as follows. We introduce the problem in more detail in Section~\ref{sec:PF}.
In Section~\ref{sec:2} we reformulate phase retrieval by changing the variable from the signal itself to the correlation of the signal, and discuss two ways to characterize a correlation sequence, leading to convex optimization problems. Extracting signals from a given correlation sequence, known as \emph{spectral factorization}, is briefly reviewed in Section~\ref{sec:3}, where we also review the notion of \emph{minimum phase}. We then propose a new measurement system based on transforming an arbitrary signal into a minimum phase one, via a very simple impulse adding technique, in Section~\ref{sec:add_impulse}.  \eldar{The highly scalable CoRK} algorithm is developed in Section~\ref{sec:5} to recover the signal, providing the ability to optimally solve large-scale 1D Fourier phase retrieval problems very efficiently. We summarize the proposed method in Section~\ref{sec:6}, with some discussion on how to modify the approach if we know that the sought signal is real. Computer simulations are provided in Section~\ref{sec:7} and show superiority of our proposed techniques both in terms of accuracy and efficiency. We finally conclude the paper in Section~\ref{sec:8}.

%{\color{blue}A conference version of part of this work appears in~\cite{HuaEldSid:ICASSP2016}. In %addition to~\cite{HuaEldSid:ICASSP2016}, this journal version provides an alternative formulation %different from the SDP based one, together with a much more scalable algorithm using the alternative %formulation. Moreover, several alternative ways to implement the proposed new measurement system are %discussed, together with an interesting comparison to the widely used \emph{holography}, suggesting %that the new measuring technique can be easily implemented in practice. More extensive simulations %are also performed.}

Throughout the paper, indices for vectors and matrices start at $0$, so that the first entry of a vector $\x$ is $x_0$, and the upper-left-most entry of a matrix $\X$ is $X_{00}$.
The superscript \eldar{$(\cdot)^*$} denotes element-wise conjugate (without transpose), and \eldar{$(\cdot)^H$} denotes Hermitian transpose of a vector or a matrix.

\section{The Phase Retrieval Problem}
\label{sec:PF}

In the phase retrieval problem, we are interested in estimating a signal $\x\in\C^N$ from the squared magnitude of its Fourier transform.
The \emph{discrete-time Fourier transform} (DTFT) of a vector $\x\in\C^N$ is a trigonometric polynomial in $\omega$ defined as
\[
X(e^{j\omega}) = \sum_{n=0}^{N-1}x_n e^{-j\omega n},
\]
and is periodic with period $2\pi$. In practice it is easier to obtain samples from the continuous function $X(e^{j\omega})$, which leads to the $M$-point \emph{discrete Fourier transform} (DFT) of $\x$,
{\color{blue}where $M\geq N$,} 
if we sample at points
\[
\omega = 0, \frac{2\pi}{M}, ..., \frac{2\pi(M-1)}{M}.
\]
This operation is equivalent to the matrix-vector multiplication $\F_M\x$, where $\F_M$ is the first $N$ columns of the $M$-point DFT matrix, i.e.,
\[
\F_M =
\begin{bmatrix}
1 & 1 & 1 & \cdots & 1 \\
1 & \phi & \phi^2 & \cdots & \phi^{N-1} \\
\vdots & \vdots & \vdots & \ddots & \vdots \\
1 & \phi^{M-1} & \phi^{2(M-1)} & \cdots & \phi^{(N-1)(M-1)} \\
\end{bmatrix},
\]
with $\phi=e^{-j2\pi/M}$. The matrix-vector multiplication $\F_M\x$ can be carried out efficiently via the fast Fourier transform (FFT) algorithm with complexity $O(M\log M)$, unlike the general case which takes $O(MN)$ flops to compute.

With this notation, our measurements are given by
\begin{equation}\label{prob}
\b= |\F_M\x|^2+\bm{w},
\end{equation}
where $\bm{w}$ is a noise vector and $|\cdot|^2$ is taken element-wise.
To recover $\x$ from $\b$ we consider a least-squares cost (which coincides with the maximum likelihood criterion assuming Gaussian noise):
\begin{equation}\label{prob:1}
\minimize_{\x\in\C^N}~~ \Big\| \b - |\F_M\x|^2 \Big\|^2.
\end{equation}
The most popular methods for solving (\ref{prob:1}) are the Gershburg-Saxton (GS) and Fienup's algorithms. Both techniques start with the noiseless scenario $\b = |\F_M\x|^2$, and reformulate it as the following feasibility problem by increasing the dimension of $\x$ to $M$ and then imposing an additional \emph{compact support} constraint:
\begin{align*}
\text{find}~~ & \x\in\C^M \\
\text{such that}~~ & \b = |\F_M\x|^2 \\
 & x_n = 0, ~~n = N, N+1, ..., M-1.
\end{align*}
It is easy to derive projections onto the two individual sets of equality constraints, but not both. Therefore, one can apply alternating projections, which leads to the GS algorithm, or Dykstra's alternating projections, which results in Fienup's algorithm when the step size is set to one~\cite{elser2013direct}. 

Due to the non-convexity of the quadratic equations, neither algorithm is guaranteed to find a solution; nonetheless, Fienup's approach has been observed to work successfully in  converging to a point that satisfies both set of constraints. When the measurements are corrupted by noise, Fienup's algorithm does not in general converge. An alternative interpretation of the GS algorithm shows that it monotonically decreases the cost function of the following optimization problem (which is different from \eqref{prob:1}),
\begin{equation}\label{prob:GS}
\begin{aligned}
\minimize_{\x\in\C^N,\bm{\psi}\in\C^M}~ & \left\|\diag{\sqrt{\b}}\bm{\psi}-\F_M\x\right\|^2 \\
\st~~ & |\psi_m| = 1, ~~m = 0, 1, ..., M-1.
\end{aligned}
\end{equation}

More recently, the general phase retrieval problem has been recognized as a non-convex \emph{quadratically constrained quadratic program} (QCQP), for which the prevailing approach is to use semidefinite relaxation \cite{luo2010semidefinite} to obtain a lower bound on the optimal value of (\ref{prob:1}). In the field of phase retrieval, this procedure is known as {\em PhaseLift} \cite{candes2013phase}. Specifically, under a Gaussian noise setting, PhaseLift solves 
\begin{equation}\label{prob:phaselift}
\minimize_{\X\in\H_+^{N}}~~ \sum_{m=0}^{M-1}\left(b_m-\tr{\f_m^{}\f_m^H\X}\right)^2 + \lambda\tr{\X},
\end{equation}
where $\H_+^{N}$ denotes the set of Hermitian positive semidefinite matrices of size $N\times N$, $\f_m^H$ is the $m$th row of $\F_M$, and the term $\lambda\tr{\X}$ is used to encourage the solution to be low-rank. Problem (\ref{prob:phaselift}) can be cast as an SDP and solved in polynomial time. If the solution of (\ref{prob:phaselift}), denoted as $\X_\star$, turns out to be rank one, then we also obtain the optimal solution of the original problem~(\ref{prob:1}) by extracting the rank one component of $\X_\star$. However, for general measurement vectors PhaseLift is not guaranteed to yield a rank one solution, especially when the measurement $\b$ is noisy. In that case PhaseLift resorts to sub-optimal solutions, for example by taking the first principal component of $\X_\star$, possibly refined by a traditional method like the GS algorithm.
An SDP relaxation for the alternative formulation (\ref{prob:GS}) is proposed in \cite{waldspurger2015phase}, and referred to as \emph{PhaseCut}.

In the next section we show that despite the nonconvexity of (\ref{prob:1}), we can find an optimal solution in polynomial time using semidefinite relaxation.
Since there is no uniqueness in 1D phase retrieval, there are many possible solutions even in the noise free setting. Among all solutions, we extract the minimum phase vector that minimizes the least-squares error. We will also suggest an alternative to the SDP formulation, \eldar{that avoids squaring the number of variables like PhaseLift.} Next we will show how any vector can be modified to be minimum phase by adding a sufficiently large impulse to it. This then paves the way to recovery of arbitrary 1D signals efficiently and provably optimally from their Fourier magnitude.

\section{Convex reformulation}\label{sec:2}
In this section we show how (\ref{prob:1}) can be optimally solved in polynomial time, despite its non-convex formulation. Similar results have been shown in the application of multicast beamforming under far-field line-of-sight propagation conditions~\cite{karipidis2007far}, and more generally for non-convex QCQPs with Toeplitz quadratics~\cite{konar2015hidden,dumitrescu2007positive}.

\subsection{\eldar{Phase retrieval using auto-correlation}}
\eldar{We begin by reformulating the Fourier phase retrieval problem in terms of the auto-correlation function.}
Consider the $m$th entry of  $|\F_M\x|^2$:
\begin{align*}
|\f_m^H\x|^2 & = \sum_{n=0}^{N-1}\phi^{nm} x_n^{} \sum_{\nu=0}^{N-1}\phi^{-\nu m} x_\nu^* \\
& = \sum_{n=0}^{N-1}\sum_{\nu=0}^{N-1}x_n^{} x_\nu^* \phi^{(n-\nu)m} \\
& = \sum_{n=0}^{N-1}\sum_{k=n+1-N}^{n} x_n^{} x_{n-k}^* \phi^{km} \\
& = \sum_{k=1-N}^{N-1}\phi^{km} \sum_{n=\max(k,0)}^{\min(N-1+k,N-1)}x_n^{} x_{n-k}^*\\
& = \sum_{k=1-N}^{N-1}\phi^{km} r_k^{},
\end{align*}
where
\begin{align}\label{eq:r}
r_k^{} & =\sum_{n=\max(k,0)}^{\min(N-1+k,N-1)}x_n^{} x_{n-k}^*, \\
	&~~ k = 1-N,...,-1,0,1,...,N-1,\nonumber
\end{align}
is the $k$-lag auto-correlation of $\x$. Let us define
\[
\tilde{\r} = [~r_{1-N}~...~r_{-1}~r_0~r_1~...~r_{N-1}~]^T.
\]

We first observe that $|\f_m^H\x|^2$, originally quadratic with respect to $\x$, is now linear in $\tilde{\r}$. Moreover, by definition
$r_k^{} = r_{-k}^*$. Removing the redundancy, we let
\[
\r = [~r_0~r_1~...~r_{N-1}~]^T,
\]
and write
\[
|\F_M\x|^2 = \real{\F_M\Itilde\r},
\]
where $\Itilde=\diag{[~1~2~2~...~2~]}$, and $\real{\cdot}$ takes the real part of its argument. We can then rewrite (\ref{prob:1}) as
\begin{equation}\label{prob:conv0}
\begin{aligned}
\minimize_{\r\in\C^N}~~ & \left\|\b - \real{\F_M\Itilde\r}\right\|^2 \\
\st~~ & \r \text{ is a finite auto-correlation sequence.}
\end{aligned}
\end{equation}
The abstract constraint imposed on $\r$ in (\ref{prob:conv0}) is to ensure that there exists a vector $\x\in\C^N$ such that (\ref{eq:r}) holds.

The representation in (\ref{eq:r}) of the auto-correlation sequence is nonconvex {\color{blue}(in $\r$ and $\x$ jointly)}. In the next subsection we consider convex reformulations of this constraint based on \cite{alkire2002convex}.
In Section~\ref{sec:3} we discuss how to obtain $\x$ from $\r$ in a computationally efficient way.

\subsection{Approximate characterization of a finite auto-correlation sequence}

Denote the DTFT of $\tilde{\r}$ by
\[
R(e^{j\omega}) = \sum_{k=1-N}^{N-1}r_k e^{-j\omega k}.
\]
Since $\tilde{\r}$ is conjugate symmetric, $R(e^{j\omega})$ is real for all $\omega\in[0,2\pi]$. The sequence $\tilde{\r}$ (or equivalently $\r$) is a finite auto-correlation sequence if and only if~\cite{krein1977markov}
\begin{equation}\label{eq:r1}
R(e^{j\omega}) \geq 0, ~~ \omega\in[0,2\pi],
\end{equation}
which is a collection of infinitely many linear inequalities. A natural way to approximately satisfy (\ref{eq:r1}) is to sample this infinite set of inequalities, and replace (\ref{eq:r1}) by a finite (but large) set of $L$ linear inequalities
\[
R(e^{j2\pi l/L}) \geq 0, \quad l = 0,1,...,L-1.
\]
In matrix form, these constraints can be written as (similar to the expression for $|\F_M\x|^2$ that we derived before)
\begin{equation*}
\real{\F_L\Itilde\r} \geq 0,
\end{equation*}
where $\F_L$ is the first $N$ columns of the $L$-point DFT matrix.

Using this result, we can approximately formulate problem~(\ref{prob:conv0}) as
\begin{equation}\label{prob:conv1}
\begin{aligned}
\minimize_{\r\in\C^N}~~ & \left\|\b - \real{\F_M\Itilde\r}\right\|^2 \\
\st~~ & \real{\F_L\Itilde\r} \geq 0,
\end{aligned}
\end{equation}
with a sufficiently large $L$.
In practice, the approximate formulation works very well for $L\geq 20N$, a modest (and typically a constant times) increase in the signal dimension. However, it is not exact---one can satisfy the constraints in~(\ref{prob:conv1}), yet still have $R(e^{j\omega})<0$ for some $\omega$.

\subsection{Exact parameterization of a finite auto-correlation sequence}

It turns out that it is possible to characterize the infinite set of inequalities (\ref{eq:r1}) through a finite representation, via an auxiliary positive semidefinite matrix. One way is to use a $N\times N$ positive semidefinite matrix to parameterize $\r${\color{blue}\cite{dumitrescu2007positive}}
\begin{align*}
r_k & = \tr{\T_k\X}, ~~k = 0,1,...,N-1,\\
\X  & \succeq 0,
\end{align*}
where $\T_k$ is the $k$th elementary Toeplitz matrix of appropriate size, with ones on the $k$th sub-diagonal, and zeros elsewhere, and $\T_0=\eye$. Another way, which can be derived from the Kalman-Yakubovich-Popov (KYP) lemma~\cite{wu1996fir}, uses a $(N-1)\times(N-1)$ matrix $\bm{P}$ to characterize $\r$ as a finite correlation sequence
\[
\begin{bmatrix}
r_0 & \r_{1:N-1}^H \\
\r_{1:N-1} & \bm{P}
\end{bmatrix}
-
\begin{bmatrix}
\bm{P} & \bm{0} \\
\bm{0} & 0
\end{bmatrix}
\succeq 0,
\]
where $\r_{1:N-1} = [~r_1~r_2~...~r_{N-1}~]^T$. It can be shown that the two formulations are equivalent~\cite{dumitrescu2001parameterization}.

Adopting the first characterization, we may rewrite (\ref{prob:conv0}) as
\begin{equation}\label{prob:conv2}
\begin{aligned}
\minimize_{\r\in\C^N,\X\in\H^N_+}~~ & \left\|\b - \real{\F_M\Itilde\r}\right\|^2 \\
\st~~~~ & r_k = \tr{\T_k\X}, ~~k = 0, 1, ..., N-1.
\end{aligned}
\end{equation}
It is easy to see that we can actually eliminate the variable $\r$ from (\ref{prob:conv2}), ending up with the PhaseLift formulation (\ref{prob:phaselift}) without the trace regularization.
This result is proven in Appendix~\ref{appx:sdp}.

The implication behind the above analysis is that, even though PhaseLift does not produce a rank one solution in general\footnote{In fact, it has been shown that if the interior-point algorithm is used to solve an SDP, it will always generate a solution that is of maximal rank~\cite{luo2010semidefinite}.}, there always exists a point in $\C^N$ that attains the cost provided by the SDP relaxation. In other words, the seemingly non-convex problem (\ref{prob:1}) can be solved in polynomial time.

The trace parameterization based formulation (\ref{prob:conv2}) exactly characterizes (\ref{prob:conv0}), with the price that now the problem dimension is $O(N^2)$, a significant increase comparing to that of (\ref{prob:conv1}). In Section~\ref{sec:5} we will derive an efficient alternative based on ADMM. In the next section we show how to extract a vector $\x$ from the auto-correlation $\r$ that solves (\ref{prob:conv2}).

\section{Spectral factorization}\label{sec:3}
After we solve (\ref{prob:conv0}), either approximately via (\ref{prob:conv1}) or exactly by (\ref{prob:conv2}), we obtain the auto-correlation of the optimal solution of the original problem (\ref{prob:1}). The remaining question is how to find a vector $\x$ that generates such an auto-correlation, as defined in (\ref{eq:r}). This is a classical problem known as \emph{spectral factorization} (SF) in signal processing and control. There exist extensive surveys on methods and algorithms that solve this problem, e.g., \cite{sayed2001survey}, \cite[Appendix]{wu1999fir}, and \cite[Appendix~B]{dumitrescu2007positive}. In this section, we first derive an algebraic solution for SF, which also explains the non-uniqueness of 1D Fourier phase retrieval, and reviews the notion of \emph{minimum phase}. We then \eldar{survey} two practical methods for SF, which we will rely on when developing \eldar{CoRK,} a highly scalable phase-retrieval algorithm.

\subsection{Algebraic solution}\label{sec:sf}
The DFT of an arbitrary signal $\x\in\C^N$ can be obtained by sampling its $z$-transform on the unit circle $|z|=1$. Let $X(z)$ be the $z$-transform of $\x$, which is a polynomial of order $N-1$. It can be written in factored form as
\begin{equation}\label{eq:zt}
X(z) = \sum_{n=0}^{N-1}x_nz^{-n} = x_0\prod_{n=1}^{N-1}(1-\xi_nz^{-1}),
\end{equation}
where $\xi_1, ..., \xi_{N-1}$ are the zeros (roots) of the polynomial $X(z)$. The quadratic measurements can similarly be interpreted as sampled from the $z$-transform of $\tilde{\r}$ defined as
\begin{align*}
R(z) & = |X(z)|^2 = X(z^{})X^*(1/z^*) \\
& = |x_0|^2\prod_{n=1}^{N-1}(1-\xi_n^{}z^{-1})(1-\xi_n^*z^{}).
\end{align*}

As we can see, the zeros of $R(z)$ always come in conjugate reciprocal pairs. Therefore, given $R(z)$, we cannot determine whether a zero $\xi$ or its conjugate reciprocal $(\xi^*)^{-1}$ is a root of $X(z)$, which is the reason that $\x$ cannot be reconstructed from $R(z)$. In other words, for a given signal $\x$, we can find the zeros of its {$z$-transform} as in (\ref{eq:zt}), take the conjugate reciprocal of some of them, and then take the inverse $z$-transform to obtain another signal $\y$\eldar{\cite{beinert2015ambiguities}}. If we re-scale $\y$ to have the same $\ell_2$ norm as $\x$, then it is easy to verify that
\[
|\F_M\x|^2 = |\F_M\y|^2,
\]
no matter how large $M$ is, even though clearly $\x$ and $\y$ are not equal.

Traditionally, this problem is often seen in design problems where uniqueness is not important, e.g., FIR filter design~\cite{wu1999fir} and far-field multicast beamforming~\cite{karipidis2007far}. There, it is natural (from the maximal energy dissipation point of view) to pick the zeros to lie within the unit circle\footnote{If a zero $\xi$ lie exactly on the unit circle, then $(\xi^*)^{-1}=\xi$, meaning $R(z)$ has a double root at that point, so we simply pick one as the zero for the signal.}, yielding a so-called {\em minimum phase} signal.

The above analysis provides a direct method for SF. For a given auto-correlation sequence $\r$, we first calculate the roots of the polynomial
\[
R(z) = r_0 + \sum_{k=1}^{N-1}\left(r_k^{}z^{-k}+r_k^*z^k\right).
\]
Because $\r$ is a valid correlation sequence, the roots come in conjugate reciprocal pairs. Therefore, we can pick the $N-1$ roots that are inside the unit circle, expand the expression, and then scale it to have $\ell_2$ norm equal to $\sqrt{r_0}$.

Numerically, the roots of $R(z)$ can be found by calculating the eigenvalues of the following companion matrix:
\[
\begin{bmatrix}
0 & 1 & 0 & \cdots & 0 \\
0 & 0 & 1 &  & \vdots \\
\vdots & \vdots & & \ddots & 0 \\
0 & 0 & \cdots & 0 & 1 \\
-\frac{r_{N-1}}{r^*_{N-1}} & \cdots & -\frac{r_0}{r^*_{N-1}} & \cdots & -\frac{r^*_{N-2}}{r^*_{N-1}} \\
\end{bmatrix}.
\]
However, when expanding the factored form to get the coefficients of the polynomial, the procedure is very sensitive to roundoff error, which quickly becomes significant as $N$ approaches $64$. %One way to circumvent this is to use the fact that the eigenvectors of the companion matrix are Vandermonde vectors generated by the corresponding eigenvalues. Therefore we can calculate the companion matrix corresponding to the minimum phase signal from its eigendecomposition, or Jordan canonical form if there are multiple roots and/or zero roots.

\subsection{SDP-based method for SF}\label{sec:sdp}
To obtain a more stable method numerically, we can use an SDP approach to SF.

For a valid correlation $\r$, it is shown in \cite[Chapter~2.6.1]{dumitrescu2007positive} that the solution of the following SDP
\begin{equation}\label{prob:sdp4sf}
\begin{aligned}
\maximize_{\X\in\H_+^{N}}~~ & X_{00} \\
\st~~~ & r_k = \tr{\T_k\X},\\
		& k=0,1,...,N-1,
\end{aligned}
\end{equation}
is always rank one. In addition, its rank one component generates the given correlation and is minimum phase. Algorithms for SDP are numerically stable, although the complexity could be high if we use a general-purpose SDP solver.

The constraints in (\ref{prob:sdp4sf}) are the same as the convex reformulation (\ref{prob:conv2}), which according to Appendix~\ref{appx:sdp} is equivalent to PhaseLift without the trace regularization.
Therefore, we may consider instead solving the problem
\begin{equation}\label{prob:phaselift1}
\minimize_{\X\in\H_+^{N}}~~ \sum_{m=0}^{M-1}\left(b_m-\tr{\f_m^{}\f_m^H\X}\right)^2 - \lambda X_{00}.
\end{equation}
In Appendix~\ref{appx:phaselift-sf} we show that for a suitable value of $\lambda$, the solution of (\ref{prob:phaselift1}) is guaranteed to be rank one, and it attains the optimal least-squares error.

The formulation (\ref{prob:phaselift1}) allows to use customized SDP solvers to obtain a solution efficiently. For example, PhaseLift was originally solved using TFOCS~\cite{becker2011templates}, which applies various modern first-order methods to minimize convex functions composed of a smooth part and a non-smooth part that has a simple proximity operator. For problem (\ref{prob:phaselift1}), the cost function is the smooth part, and the non-smooth part is the indicator function for the cone of semidefinite matrices.

It is now widely accepted that the trace of a semidefinite matrix, or nuclear norm for a general matrix, encourages the solution to be low rank~\cite{recht2010guaranteed}. However, interestingly, in our setting, the correct regularization is in fact $-\lambda X_{00}$ which guarantees the solution to be rank one for an appropriate choice of $\lambda$.

\subsection{Kolmogorov's method for SF}\label{sec:kolmo}
Kolmogorov proposed the following method for SF that is both efficient and numerically stable \cite{wu1996fir}. In what follows, we assume that $R(z)$ contains no zeros on the unit circle.

Consider taking the logarithm of the $z$-transform of $\x$. Every zero (and pole, which we do not have since $\x$ is finite length) of $X(z)$ then becomes a pole of $\log X(z)$. Assuming the roots of $X(z)$ lie strictly inside the unit circle, this implies that there exists a region of convergence (ROC) for $\log X(z)$ that contains the unit circle $|z|=1$ and infinity $|z|=\infty$. This in turn means that $\log X(z)$ is unilateral
\[
\log X(z) = \sum_{n=0}^{\infty}\alpha_nz^{-n}.
\]
Evaluating $\log X(z)$ at $z=e^{j\omega}$, yields
\begin{align*}
\real{\log X(e^{j\omega})} &= ~~\sum_{n=0}^{\infty}\alpha_n\cos\omega n, \\
\imag{\log X(e^{j\omega})} &=  -\sum_{n=0}^{\infty}\alpha_n\sin\omega n,
\end{align*}
implying that $\real{\log X(e^{j\omega})}$ and $\imag{\log X(e^{j\omega})}$ are Hilbert transform pairs. Moreover, since we are given $R(z)=|X(z)|^2$, we have that
\[
\real{\log X(e^{j\omega})}=\frac{1}{2}\log R(e^{j\omega}).
\]
We can therefore calculate $\imag{\log X(e^{j\omega})}$ from the Hilbert transform of $\real{\log X(e^{j\omega})}$, and reversely reconstruct a signal $\x$ that generates the given correlation $\r$ and is minimum phase.

In practice, all of the aforementioned transforms may be well approximated by a DFT with sufficiently large length~$L$. The detailed procedure of Kolmogorov's method then becomes:
\begin{enumerate}
\item compute the real part
\[
\bm{\gamma} = \frac{1}{2}\log\real{\F_L\Itilde\r};
\]
\item compute the imaginary part by taking the Hilbert transform of $\bm{\gamma}$, approximated by a DFT
\begin{align*}
\bm{\phi} &= \F_L\bm{\gamma},\\
\varphi_n &= \begin{cases}
0, & n = 0, L/2, \\
-j\phi_n, & n = 1,2,...,L/2-1, \\
 j\phi_n, & n = L/2+1, ..., L-1,
\end{cases}\\
\bm{\eta} &= \frac{1}{L}\F_L^H\bm{\varphi};
\end{align*}
\item compute $(1/L)\F_L^H\exp(\bm{\gamma}-j\bm{\eta})$, and take the first $N$ elements as the output.
\end{enumerate}

Kolmogorov's method generates a valid output as long as in the first step we have that $\real{\F_L\Itilde\r}\geq 0$, so that $\bm{\gamma}$ is real. This fits well with the approximate formulation (\ref{prob:conv1}), since if we choose the same $L$ for both (\ref{prob:conv1}) and Kolmogorov's method, then it is guaranteed to be able to run without encountering a syntax error. Obviously, to obtain a solution with high accuracy, we need $L$ to be sufficiently large; empirically we found $L\geq 20N$ to be good enough in practice. Computationally, this approach requires computing four $L$-point (inverse-) FFTs with complexity $O(L\log L)$, a very light computation burden even for large $L$.

\section{A new measurement system}\label{sec:add_impulse}
So far we have seen that an arbitrary signal cannot be uniquely determined from the magnitude of its over-sampled 1D Fourier measurements, because there always exists a minimum phase signal that yields the same measurements. Nonetheless, a least-squares estimate can be efficiently calculated by solving either (\ref{prob:conv1}) or (\ref{prob:conv2}) followed by spectral factorization, leading to an optimal solution that is minimum phase. If the true signal $\x$ is indeed minimum phase, then we can optimally estimate it in polynomial-time. However, the minimum phase property is not a natural assumption to impose on a signal in general.

We propose to resolve this ambiguity by deliberately making the signal minimum phase before taking the quadratic measurements, so that the augmented signal is uniquely identified in polynomial time. The true signal is then recovered easily by reversing this operation.

For an arbitrary complex signal $\s$, we \eldar{suggest adding} $\delta$ in front of $\s$ before taking measurements, where $\delta$ satisfies that $|\delta|\!\geq\!\|\s\|_1$. Denote the augmented signal as $\smin$. The following proposition shows that $\smin$ is minimum phase.
\begin{proposition}\label{thm:minphase}
Consider an arbitrary complex signal $$\s=[~s_0~s_1~...~s_{N-1}~]^T.$$ Then the augmented signal
\begin{equation}\label{eq:smin}
\smin=[~\delta~s_0~...~s_{N-1}~]^T,
\end{equation}
where $|\delta|\geq\|\s\|_1$, is minimum phase.
\end{proposition}
\begin{proof}
To establish the result we need to show that the zeros of the $z$-transform of~$\smin$
\[
\delta + s_0z^{-1} + ... + s_{N-1}z^{-N},
\]
or equivalently the roots of the polynomial
\[
V(z)=z^{N} + \frac{s_0}{\delta}z^{N-1} + ... + \frac{s_{N-1}}{\delta},
\]
all lie inside the unit circle.
To this end, we rely on the following lemma.
\begin{lemma}\label{lmm:roots}
\cite[Theorem~1]{hirst1997bounding} Let $\zeta$ be a zero of the polynomial
\[
z^N + c_{N-1}z^{N-1} + ... + c_1z + c_0,
\]
where $c_0,...,c_{N-1}\in\C$ and $N$ is a positive integer. Then
\[
|\zeta| \leq \max\left\{1,\sum_{n=0}^{N-1}|c_n|\right\}.
\]
\end{lemma}
Substituting the coefficients of $V(z)$ into the inequality in Lemma~\ref{lmm:roots} establishes the result.
\end{proof}

Conceptually, the approach we propose is very simple: all we need is a way to over-estimate the $\ell_1$ norm of the target signal, and a mechanism to insert an impulse in front of the signal before taking quadratic measurements. For example, if we assume each element in $\s$ comes from a complex Gaussian distribution with variance $\sigma^2$, then we know that the probability that the magnitude of one element exceeds $3\sigma$ is almost negligible; therefore, we can simply construct $\smin$ by setting $\delta=3\sigma N$, resulting in $\smin$ being minimum phase with very high probability.

Our approach can be used with a number of measurements $M$, as small as $2N$. Indeed, consider the equivalent reformulation (\ref{prob:conv0}), in which the measurements $\b$ are linear with respect to $\r$. From elementary linear algebra, we know that $N$ complex numbers can be uniquely determined by as few as $2N$ real linearly independent measurements, even without the specification that $\r$ is a finite correlation sequence. From a unique $\r$, a unique minimum phase $\smin$ can be determined using SF.

The impulse may also be appended at the end of the signal $\s$, resulting in a \emph{maximum phase} signal, meaning all the zeros of its $z$-transform are outside the unit circle.
\begin{proposition}\label{thm:maxphase}
For an arbitrary complex signal $$\s=[~s_0~s_1~...~s_{N-1}~]^T,$$ the augmented signal $$\smax=[~s_0~...~s_{N-1}~\delta~]^T,$$ where $|\delta|\geq\|\s\|_1$, is maximum phase.
\end{proposition}
\begin{proof}
The $z$-transform of $\smax$ is
\[
S(z)=s_0 + s_1z^{-1} + ... + s_{N-1}z^{1-N} + \delta z^{-N}.
\]
Let $\tilde{z}=z^{-1}$. Then
\[
S(\tilde{z})=s_0 + s_1\tilde{z} + ... + s_{N-1}\tilde{z}^{N-1} + \delta\tilde{z}^{N},
\]
which is a polynomial in $\tilde{z}$ and, according to Lemma~\ref{lmm:roots}, has all its roots inside the unit circle. Taking the reciprocal, this means that all the zeros of the $z$-transform of $\smax$ lie outside the unit circle.
\end{proof}
It is easy to show that for a maximum phase signal $$\smax=[~s_0~...~s_{N-1}~\delta~]^T,$$ its equivalent minimum phase signal is
\[
[~\delta^*~s_{N-1}^*~s_{N-2}^*~...~s_0^*~]^T.
\]
This means that if we measure the intensity of the Fourier transform of $\smax$ instead, we can still uniquely recover $\s$ via solving (\ref{prob:conv0}) followed by SF, take the conjugate reversal of the solution, and then delete the first element to obtain an estimate.

Furthermore, from the analysis above, we can extend our measuring technique so that the impulse is added away from the signal:
\begin{equation}\label{eq:holography1}
[~\delta~0~...~0~s_0~...~s_{N-1}~]^T,
\end{equation}
or
\begin{equation}\label{eq:holography2}
[~s_0~...~s_{N-1}~0~...~0~\delta~]^T,
\end{equation}
with an arbitrary number of zeros between $\delta$ and $\s$. The resulting augmented signal is still minimum/maximum phase, thus can be uniquely identified from the intensity of its Fourier transform. However, in this case more measurements are required for identifiability.

{\color{blue}It was briefly mentioned in~\cite{yagle1999one} \eldar{without proof} that if the signal has a strong first component, then it is minimum phase. Here we propose {\em adding} an impulse as a means of restoring identifiability for \emph{arbitrary} signals.}
This concept is very similar to {\em holography}, which was invented by Gabor in 1948~\cite{gabor1948new}, and was awarded the Nobel Prize in Physics in 1971. One form of holography, spectral interferometry, inserts an impulse exactly in the form of (\ref{eq:holography1}) or (\ref{eq:holography2}), and then measures the squared magnitude of the Fourier transform of the combined signal~\cite{takeda1982fourier,diels1985control}. Despite the similarities in obtaining the measurements, the theory behind our approach and holography is very different, in several ways:
\begin{enumerate}
\item Holography constrains the length of the zeros between $\s$ and $\delta$ to be at least $N$, the length of the signal, which is not required in our method; in fact, in terms of the number of measurements needed, the shorter the better.
\item We require $|\delta|\geq\|\s\|_1$, whereas holography does not have any restriction on the energy of the impulse.
\item We provide an optimization based framework to recover the underlying signal directly in a robust fashion by exploiting the minimum phase property which is not part of the framework in holography.
\end{enumerate}
%We are currently collaborating with Prof. Moti Segev and Maor Mutzafi from the Technion to explore extensions of the proposed measuring technique for holography.

\section{Scalable algorithms}\label{sec:5}

In this section, we briefly describe how to directly use existing PhaseLift solvers to obtain guaranteed rank one solutions efficiently. Due to the extra memory required using SDP-relaxation methods for large-scale problems, we then design a new algorithm \eldar{called CoRK} that solves the approximate formulation (\ref{prob:conv1}) with high efficiency, in terms of both time and memory.

\subsection{SDP based algorithms}
In the original PhaseLift paper \cite{candes2013phase}, the authors used TFOCS~\cite{becker2011templates} to solve the SDP~(\ref{prob:phaselift}), which is a relatively efficient way to solve large scale SDPs. As we discussed in Section~\ref{sec:sdp}, we can ensure a rank one minimum phase solution by solving (\ref{prob:phaselift1}) with an appropriate choice of $\lambda$. {\color{blue}We find such a $\lambda$} via bisection: first we solve (\ref{prob:phaselift1}) with $\lambda=0$ to get the optimal value of the fitting term; next we solve it with a $\lambda$ large enough that the solution is guaranteed to be rank one, but with a possibly larger fitting error; finally we bisect $\lambda$ until both criteria are met. Since TFOCS allows to specify initializations, after we solve (\ref{prob:phaselift1}) for the first time, subsequent evaluations can be performed much more rapidly via warm start. Moreover, there is a wide range of $\lambda$ that result in both a rank one solution and an optimal fitting error, so that the overall increase in time required is small compared to that of solving a single PhaseLift problem.

The disadvantage of SDP based algorithms is that one cannot avoid lifting the problem dimension to $O(N^2)$. If $N$ is moderately large, on the order of a few thousand, then it is very difficult to apply SDP based methods in practice.

\subsection{CoRK: An efficient ADMM method}
We now propose a new algorithm for solving the approximate formulation (\ref{prob:conv1}), which avoids the dimension lifting, and easily handles problem sizes up to millions. Since (\ref{prob:conv1}) is a convex quadratic program, there are plenty of algorithms to solve it reliably. For example, a general purpose interior-point method solves it with worst case complexity $O(N^{3.5})$ without taking any problem structure into account. Nevertheless, noticing that all of the linear operations in (\ref{prob:conv1}) involve DFT matrices, {\color{blue}we aim to further recude the complexity by exploiting the FFT operator.}

We propose solving (\ref{prob:conv1}) using the alternating direction method of multipliers (ADMM)~\cite{boyd2011distributed}. To apply ADMM, we first rewrite (\ref{prob:conv1}) by introducing an auxiliary variable $\z\in\R^L$:
\begin{equation}\label{prob:admm}
\begin{aligned}
\minimize_{\r\in\C^N,\z\in\R^L}~~ & \left\|\b-\real{\F_M\Itilde\r}\right\|^2 + {\cal I}_+(\z) \\
\st~~~ & \real{\F_L\Itilde\r} = \z,
\end{aligned}
\end{equation}
where
\[
{\cal I}_+(\z) = \begin{cases}
0, & \z \geq 0, \\
+\infty, & \text{otherwise,}
\end{cases}
\]
is the indicator function of the non-negative orthant in $\R^L$.
Treating $\r$ as the first block of variables, $\z$ as the second, and $\Re\{\F_L\Itilde\r\} = \z$ as the coupling linear equality constraint with scaled Lagrange multiplier $\bm{u}$, we arrive at the following iterative updates:
\begin{equation}\label{alg:admm}
\boxed{
\begin{aligned}
\r &\leftarrow \frac{1}{M+\rho L}\left(\F_M^H\b + \rho\F_L^H(\z-\bm{u})\right), \\
\z &\leftarrow \max\left( 0, \real{\F_L\Itilde\r} + \bm{u} \right), \\
\bm{u} &\leftarrow \bm{u} + \real{\F_L\Itilde\r} - \z,
\end{aligned}
}
\end{equation}
\eldar{where $\rho$ is a constant described below.}
The derivation of the algorithm is given in Appendix~\ref{appx:admm}.

All the operations in (\ref{alg:admm}) are taken element-wise or involve FFT computations of length $L$. This leads to a complexity $O(L\log L)$, and effectively $O(N \log N)$ if $L$ is chosen as $O(N)$.
{\color{blue}In contrast, a naive implementation of ADMM for quadratic programming requires solving a least-squares problem in every iteration~\cite{ghadimi2015optimal}, leading to an $O(MN^2)$ per-iteration complexity, which is significantly higher.}

In terms of convergence rate, it is shown in~\cite[Theorem 3]{ghadimi2015optimal} that our ADMM approach will converge linearly, for all values of $\rho > 0$. The same reference provides an optimal choice of $\rho$ that results in the fastest convergence rate, for the case when there are fewer inequality constraints than variables in the first block. However, this requirement unfortunately is not fulfilled for problem (\ref{prob:conv1}). Inspired by the choice of $\rho$ in \cite{ghadimi2015optimal}, we found empirically that by setting $\rho=M/L$, the proposed iterates (\ref{alg:admm}) converge very fast, and at a rate effectively independent of the problem dimension. {\color{blue}For practical purposes, the convergence rate we have achieved is good enough (typically in less than 100 iterations), but there exists pre-conditioning methods to further accelerate it~\cite{giselsson2016linear}.}

As explained in Section~\ref{sec:kolmo}, this formulation blends well with Kolmogorov's method for spectral factorization. We thus \eldar{refer to our} approach for solving the 1D Fourier phase retrieval problem (\ref{prob:1}) as the \emph{auto-\underline{co}rrelation \underline{r}etrieval -- \underline{K}olmogorov factorization} ({\bf CoRK}) algorithm.

\section{Summary of the proposed method}\label{sec:6}
Throughout the paper, we divided 1D Fourier phase retrieval into several steps (measurement, formulation, algorithms, etc.), and for each step discussed several solution methods. For clarity and practical purposes, we summarize our approach below based on specific choices we found to be efficient. 

Given an arbitrary vector $\s$, 
{\color{blue}we propose the following steps if it is possible to insert an impulse to the signal before measuring its Fourier intensity; otherwise, only step 2 will be invoked, in which case we are guaranteed to solve the maximum likelihood formulation optimally, but the estimation error \eldar{may still be large}.}
\begin{enumerate}
\item Construct $\smin$ by inserting $\delta$ in front of $\s$, i.e.,
\[
\smin = [~\delta~s_0~s_1~...~s_{N-1}~]^T,
\]
such that $|\delta|>\|\s\|_1$. Take the $M$-point DFT of $\smin$, where $M>2N$, and measure its squared magnitude $\b=|\F_M\smin|^2$.

\item Apply CoRK as follows:
\begin{enumerate}
\item Formulate the least-squares problem with respect to the auto-correlation of $\smin$ as in (\ref{prob:conv1}),
\begin{align*}
\minimize_{\r\in\C^{N+1}}~~ & \left\|\b - \real{\F_M\Itilde\r}\right\|^2 \\
\st~~ & \real{\F_L\Itilde\r} \geq 0,
\end{align*}
by picking $L$ as the smallest power of $2$ that is greater than $32N$. Solve this problem using the iterative algorithm (\ref{alg:admm}), repeated here by setting $\rho=M/L$:
\begin{align*}
\r &\leftarrow \frac{1}{2}\left(\frac{1}{M}\F_M^H\b + \frac{1}{L}\F_L^H(\z-\bm{u})\right), \\
\z &\leftarrow \max\left( 0, \real{\F_L\Itilde\r} + \bm{u} \right), \\
\bm{u} &\leftarrow \bm{u} + \real{\F_L\Itilde\r} - \z.
\end{align*}

\item Extract the minimum phase signal that generates the correlation $\r$ by Kolmogorov's method, using the same $L$, as follows (also given in Section~\ref{sec:kolmo}):
\begin{align*}
\bm{\gamma} &= \frac{1}{2}\log\real{\F_L\Itilde\r},\\
\bm{\phi} &= \F_L\bm{\gamma},\\
{\color{blue}\varphi_n} &= \begin{cases}
0, & n = 0, L/2, \\
-j\phi_n, & n = 1,2,...,L/2-1, \\
 j\phi_n, & n = L/2+1, ..., L-1,
\end{cases}\\
\bm{\eta} &= \frac{1}{L}\F_L^H\bm{\varphi},\\
\x &= (1/L)\F_L^H\exp(\bm{\gamma}-j\bm{\eta}).
\end{align*}
\end{enumerate}

\item Obtain an estimate $\hat{\s}$ via
\[
\hat{s}_n = x_{n+1}, \quad n = 0, 1, ..., N-1.
\]
\end{enumerate}

Often, we have prior information that the signal of interest $\s$ is real. This implies that the auto-correlation is also real so that we need to restrict the domain of $\r$ and/or $\X$ in problem (\ref{prob:conv1}) or (\ref{prob:phaselift1}) to be real. For the $z$-transform of a real signal, the zeros are either real or come in conjugate pairs, which does not help the non-uniqueness of the solution if we directly measure the 1D Fourier magnitude. In Kolmogorov's SF method, if the input $\r$ is real and a valid correlation, then $\bm{\gamma}+j\bm{\eta}$ is conjugate symmetric, so that it is guaranteed to output a real valued signal. As for the new measurement system, all the claims made for constructing minimum/maximum phase signals still hold, \eldar{when} the signal is restricted to be real. In Appendix~\ref{appx:admm} we show how to modify the ADMM method (\ref{alg:admm}) to accommodate real signals by adding an additional projection to the real domain in the $\r$ update.

When $\x$ is real, the DFT has conjugate symmetry
\[
X(e^{-j2\pi m/M}) = X^*(e^{-j2\pi(M-m)/M}),
\]
meaning for the squared magnitude, $b_m = b_{M-m}$. Therefore, if we take $M$ samples between $[0,2\pi]$, then only the first $M/2$ measurements provide useful information. Consequently, we still need $M>2N$ measurements sampled between $[0,2\pi]$ to ensure identifiability.

\section{Simulations}\label{sec:7}
We now demonstrate the algorithms and the new measurement system we proposed via simulations. We first show that both algorithms, iteratively solving (\ref{prob:phaselift1}) while bisecting $\lambda$ and solving (\ref{prob:conv1}) followed by Kolmogorov's method, are able to solve the original problem (\ref{prob:1}) optimally, meaning the cost attains the lower bound given by PhaseLift. Then, we show that \eldar{applying} the new measurement system proposed in Section~\ref{sec:add_impulse} we can uniquely identify an arbitrary 1D signal, whereas directly measuring the over-sampled Fourier intensity does not recover the original input, regardless of the algorithms being used. All simulations are performed in MATLAB on a Linux machine.

\subsection{Minimizing the least-squares error}\label{sec:sim1}
We first test the effectiveness of the proposed algorithms on random problem instances. Fixing $N=128$, we randomly set $M$ as an integer between $[2N,8N]$, and generate $\b$ from an i.i.d. uniform distribution between $[0,1]$. We compare the following algorithms:
\begin{itemize}
\item {\bf PhaseLift.} Using TFOCS~\cite{becker2011templates} to solve problem (\ref{prob:phaselift}) with $\lambda=0$. This in general does not give a rank one solution, therefore only serves as a theoretical lower bound on the minimal least-squares error (\ref{prob:1}).
\item {\bf PhaseLift-PC.} The leading principal component of the plain PhaseLift solution.
\item {\bf PhaseLift-SF.} Iteratively solving (\ref{prob:phaselift1}) while bisecting $\lambda$ until an equivalent rank one solution for (\ref{prob:phaselift}) is found, again using TFOCS. Each time TFOCS is initialized with the solution from the previous iteration for faster convergence. As we proved, this method is guaranteed to achieve the lower bound given by PhaseLift. %{\bf but that's not what we see in Fig. 1}.
\item {\bf CoRK.} Proposed method summarized in the second step of Section~\ref{sec:7}. For both steps, $L$ is set to be the smallest power of $2$ that is greater than $32N$.
\item {\bf Fienup-GS.} Fienup's algorithm~\cite{fienup1978reconstruction} \eldar{with} 1000 iterations, then refined by the GS algorithm until the error defined in (\ref{prob:GS}) converges.
\end{itemize}

%{\bf as we discussed it would be good to show here PhaseLift followed by rank-one based on eigenvectors. In theory, PhaseLift SF should exactly be equal to PHaseLift. why then do we have a gap?}.

The optimality gaps between the minimal error in (\ref{prob:1}) obtained by the aforementioned methods and the theoretical lower bound given by PhaseLift are shown in Fig.~\ref{fig:1}; The running time of the different methods is provided in Fig.~\ref{fig:2}, for the 100 Monte-Carlo trials we tested. As we can see, PhaseLift-SF provides an optimal rank one solution, although it takes more time compared to solving one single PhaseLift problem. Since PhaseLift-SF only provides a solution that is ``numerically'' rank one, when we take its rank one component and evaluate the actual fitting error (\ref{prob:1}), it is still a little bit away from the theoretical PhaseLift lower bound, as shown in the red circles in Fig.~\ref{fig:1}. On the other hand, the proposed CoRK method is able to approximately solve the problem with high accuracy (a lot of times even better than PhaseLift-SF) in a very small amount of time (shorter than the standard Fienup-GS algorithm). The conventional Fienup-GS method and the leading principal component of the PhaseLift solution (in general not close to rank one) do not come close to the PhaseLift lower bound.

\begin{figure}[!t]
\centering
\includegraphics[width=.7\textwidth]{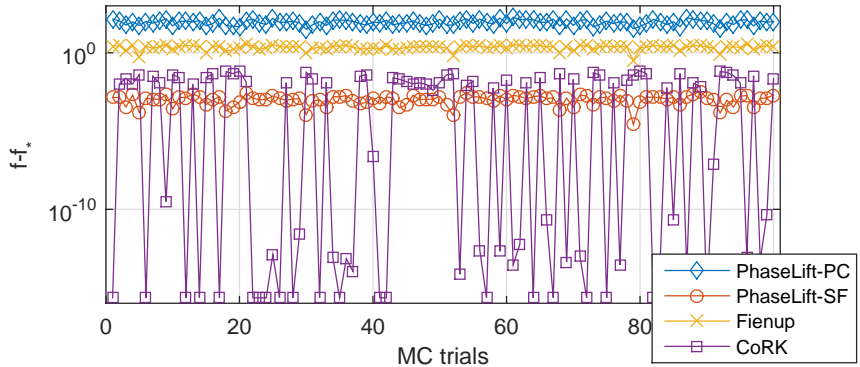}
\caption{Optimality gaps from the PhaseLift lower bound in each Monte-Carlo trial, cf. Section~\ref{sec:sim1}.}
\label{fig:1}
\end{figure}
\begin{figure}[!t]
\centering
\includegraphics[width=.7\textwidth]{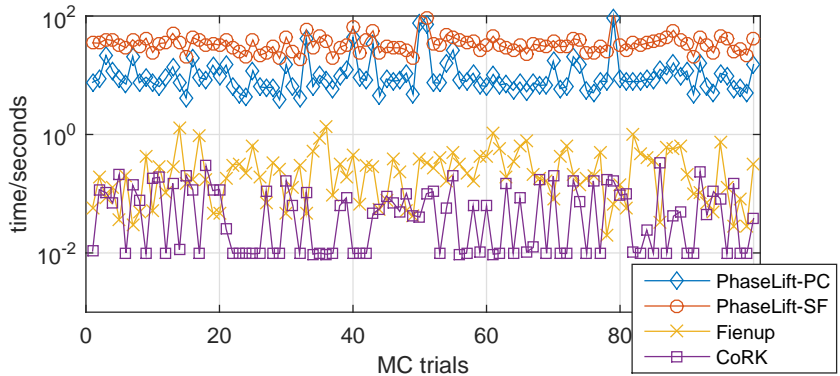}
\caption{Time consumed in each Monte-Carlo trial, cf. Section~\ref{sec:sim1}.}
\label{fig:2}
\end{figure}

\subsection{Estimation performance}\label{sec:sim2}
We now verify that our proposed measurement technique, described in Section~\ref{sec:add_impulse}, is able to recover a signal $\s$ up to global phase ambiguity. As we have shown in the previous simulation, the CoRK method performs similar to those based on PhaseLift, but with much more efficiency. Therefore, we only compare CoRK and Fienup's algorithm as baselines.
In each Monte-Carlo trial, we first set $N$ as a random integer between $[1,1024]$, and then \eldar{choose} $M$ as the smallest power of $2$ that is larger than $4N$. A random signal $\s\in\C^N$ is then generated with elements drawn from i.i.d. $\CN(0,1)$. Two kinds of measurements are collected for performance comparison:
\begin{itemize}
\item {\bf Direct:} directly measure $|\F_M\s|^2$;
\item {\bf Minimum Phase:} Construct a minimum phase signal $\smin$ as in (\ref{eq:smin}) with $\delta=3N$, and then measure $|\F_M\smin|^2$.
\end{itemize}

For the minimum phase measurements, when Fienup's algorithm is used, we add an additional step in which we use the solution to generate an auto-correlation sequence and then apply spectral factorization (Kolmogorov's method) to obtain a minimum phase signal with the same fitting error. We denote this approach by {\bf Fienup-SF}. We \eldar{employ} a simple prior that the added impulse $\delta$ is real and positive to resolve the global phase ambiguity: after obtaining the minimum phase solution, the result is first rotated so that the first entry is real and positive, and then this entry is deleted to obtain an estimate $\hat{\s}$. For direct measurements, the estimation error is defined as
\[
\min_{|\psi|=1}\|\s - \psi\hat{\s}\|^2.
\]
The estimation error in each of the 100 Monte-Carlo trials is shown in Fig.~\ref{fig:3}. We obtain perfect signal recovery when the new measurement system is used together with \eldar{the CoRK recovery algorithm. This }is never the case for direct measurements, even though the fitting error $\|\b-|\F_M^H\x|^2\|^2$ is always close to zero. For the new measuring system, CoRK obtains a solution with much higher accuracy, and lower computation time (shown in Fig.~\ref{fig:4}), compared to the widely used Fienup's algorithm.

\begin{figure}[!t]
\centering
\includegraphics[width=.7\textwidth]{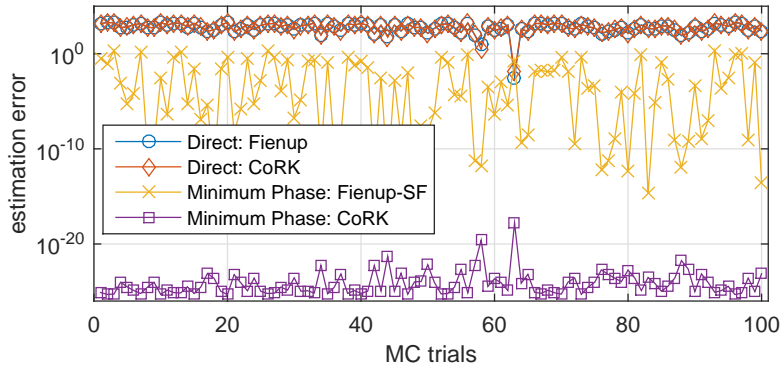}
\caption{Estimation error in each Monte-Carlo trial for the first simulation in Section~\ref{sec:sim2}.}
\label{fig:3}
\end{figure}
\begin{figure}[!t]
\centering
\includegraphics[width=.7\textwidth]{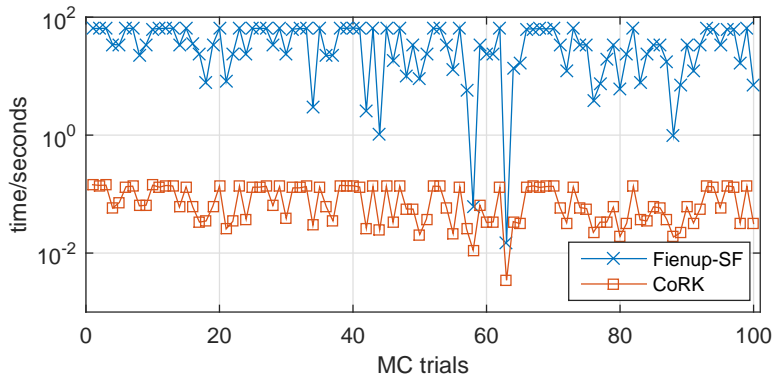}
\caption{Computation time in each Monte-Carlo trial (with the new measurement system) for the first simulation in Section~\ref{sec:sim2}.}
\label{fig:4}
\end{figure}

The new measuring system, together with the proposed CoRK method or Fienup's algorithm followed by spectral factorization, is also robust to noise, in the sense that the mean squared error (MSE) $\textbf{E}\|\s-\hat{\s}\|^2$ can essentially attain the Cram{\'e}r-Rao bound (CRB). The CRB for the phase-retrieval problem is derived in~\cite{crlb4pr}, and \eldar{is used here} to benchmark the performance of our new technique, which is the only method that can guarantee perfect signal reconstruction in the noiseless case. The CRB results in~\cite{crlb4pr} are with respect to the real and imaginary part of the signal being measured, in our case $\smin$. We therefore sum over the diagonals of the pseudo-inverse of the Fisher information matrix except for the $1$st and $n+1$st entries, which corresponds to the real and imaginary parts of $\delta$, and define that as the CRB for $\textbf{E}\|\s-\hat{\s}\|^2$. Furthermore, since the CRB is dependent on the true value of $\smin$, we show the results with respect to a fixed signal $\s$ in order to keep the CRB curve consistent with how we change one parameter setting of the simulation.

We set $N=1024$ and generate a fixed signal $\s\in\C^N$ from i.i.d. $\CN(0,1)$. Similar to the previous simulation, $\smin$ is constructed according to (\ref{eq:smin}) with $\delta=3N$, so that it is minimum phase with very high probability. White Gaussian noise with variance $\sigma^2$ is added to the squared magnitude of the Fourier transform of $\smin$. The signal-to-noise ratio (SNR) is defined as $10\log_{10}(\||\F_M\smin|^2\|^2/M\sigma^2).$
The performance is plotted in Fig.~\ref{fig:4}, where we show the normalized error $\|\s-\hat{\s}\|^2/\|\s\|^2$ versus the normalized CRB (original CRB divided by $\|\s\|^2$), averaged over 100 Monte-Carlo trials. On the left, we fix SNR$=40$dB, and increase the number of measurements $M$ from $2N$ to $16N$. On the right, we fix $M=8N$, and increase the SNR from $30$dB to $60$dB. The SNR may seem high here, but notice that most of the signal power is actually concentrated in the artificially added impulse $\delta$, so that the actual noise power is much higher comparing to that of $\s$ {\it per se}. In all cases the MSE obtained from our proposed method is able to attain the CRB sharply, even for as few as $M=2N$ measurements.

\begin{figure}[!t]
\centering
\includegraphics[width=.8\textwidth]{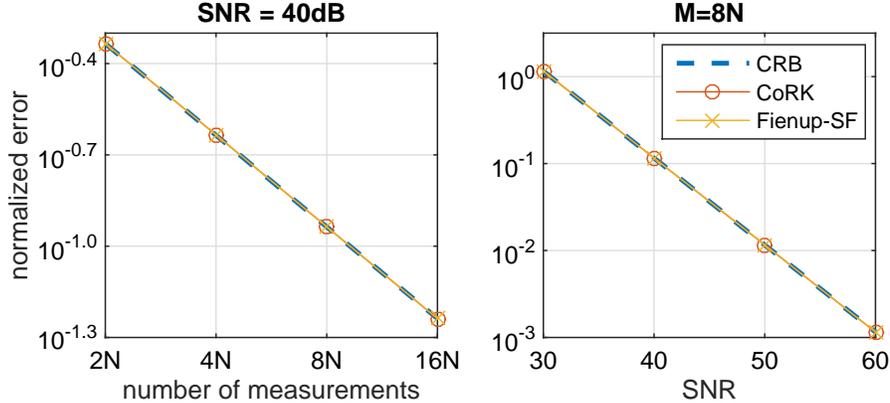}
\caption{Normalized MSE $\|\s-\hat{\s}\|^2/\|\s\|^2$ vs. normalized CRB using the new measurement system, where we increased the number of measurements on the left, and SNR on the right.}
\label{fig:5}
\end{figure}

\section{Conclusion}\label{sec:8}
We studied the phase retrieval problem with 1D Fourier measurements, a classical estimation problem that has challenged researchers for decades. Our contributions to this challenging problem are as follows:
\begin{itemize}
\item \emph{Convexity.} We showed that the 1D Fourier phase retrieval problem can be solved by a convex SDP.
\item \emph{Uniqueness.} We proposed a simple measurement technique that adds an impulse to the signal before measuring the Fourier intensities, so that any signal can be uniquely identified.
\item \emph{Algorithm.} We developed CoRK -- a highly scalable algorithm to solve this problem, which was shown in our simulations to outperform all prior art.
\end{itemize}

{\color{blue}In terms of future research directions, it is interesting to investigate how to incorporate constraints such as sparsity, non-negativity, or bandlimitedness into the auto-correlation parametrization. It is also tempting to consider extending our hidden convexity result and fast algorithm from 1D to other Fourier-based phase retrieval problems, for example the 2D case, which is identifiable, but difficult to solve.
}

\appendix

\section{Equivalence of PhaseLift and Problem~(\ref{prob:conv2})}\label{appx:sdp}
We show that the exact convex reformulation (\ref{prob:conv2}) is in fact equivalent to PhaseLift (\ref{prob:phaselift}) without the trace regularization if we eliminate the variable $\r$ in (\ref{prob:conv2}). A similar claim is given in~\cite{konar2015hidden} for the more general Toeplitz QCQP, but we show this again here in our context for completeness.

Consider the $m$th entry in the vector $\Re\{\F_M\Itilde\r\}$, and replace $\r$ by its trace parameterization
\[
r_k = \tr{\T_k\X},~ k = 0, 1, ..., N-1.
\]
We then have that
\begin{eqnarray}
\lefteqn{\real{\f_m^H\Itilde\r} = r_0 + 2\sum_{k=1}^{N-1}\real{\phi^{km}r_k}} \nonumber \\
&= &r_0 + \sum_{k=1}^{N-1}\left(\phi^{km}r_k^{} + \phi^{-km}r_k^*\right) \nonumber \\
&= &\tr{\T_0\X} + \sum_{k=1}^{N-1}\left(\phi^{km}\tr{\T_k\X}+ \phi^{-km}\tr{\T_k^T\X} \right)\nonumber \\
&= & \tr{\f_m^{}\f_m^H\X},
\end{eqnarray}
where the last step is a result of the fact that
\[
\T_0 + \sum_{k=1}^{N-1}\left(\phi^{km}\T_k^{} + \phi^{-km}\T_k^T\right) = \f_m^{}\f_m^H.
\]

We  now eliminate $\r$ in (\ref{prob:conv2}) as follows:
\begin{eqnarray}
  \lefteqn{\left\|\b - \real{\F_M\Itilde\r}\right\|^2
 = \sum_{m=0}^{M-1}\left(b_m-\real{\f_m^H\Itilde\r}\right)^2 }\nonumber \\
& =& \sum_{m=0}^{M-1}\left(b_m-\tr{\f_m^{}\f_m^H\X}\right)^2,
\end{eqnarray}
which is exactly the cost function for PhaseLift~(\ref{prob:phaselift}) without the trace regularization.

\section{Combining PhaseLift with spectral factorization}\label{appx:phaselift-sf}
As we explained in the paper, 1D Fourier phase retrieval can be solved exactly using SDP in the following two steps. First, solve problem (\ref{prob:conv2}) (repeated here),
\begin{equation}\label{eq:a1}
\begin{aligned}
\minimize_{\r\in\C^N,\X\in\H^N_+}~~ & \left\|\b - \real{\F_M\Itilde\r}\right\|^2 \\
\st~~~~ & r_k = \tr{\T_k\X}, ~~k = 0, 1, ..., N-1.
\end{aligned}
\end{equation}
Denote $\r_\star$ as the optimal solution of (\ref{eq:a1}), which is the unique solution since the cost of (\ref{eq:a1}) is strongly convex with respect to $\r$. Next, perform an SDP-based spectral factorization by computing the solution to
\begin{equation}\label{eq:a2}
\begin{aligned}
\maximize_{\X\in\H_+^{N}}~~ & X_{00} \\
\st~~~ & r_{\star k} = \tr{\T_k\X},  k = 0, 1, ..., N-1,
\end{aligned}
\end{equation}
and let the optimal solution of (\ref{eq:a2}) be $\X_\star$. Note that $\X_\star$ is the unique solution for (\ref{eq:a2}), and it is rank one~\cite{dumitrescu2007positive}. Here we want to show that these two steps can actually be combined into one, as shown in the following proposition.
\begin{proposition}\label{ppst:phaselift-sf}
Consider the following SDP
\begin{equation}\label{eq:a3}
\minimize_{\X\in\H_+^{N}}~~ \sum_{m=0}^{M-1}\left(b_m-\tr{\f_m^{}\f_m^H\X}\right)^2 - \lambda X_{00}.
\end{equation}
There exists some positive $\lambda$ such that the solution of (\ref{eq:a3}) is guaranteed to be $\X_\star$, which is also the solution of (\ref{eq:a2}) with $\r_\star$ being the solution of (\ref{eq:a1}).
\end{proposition}
\begin{proof}
Notice that $(\r_\star,\X_\star)$ is a feasible solution for (\ref{eq:a1}), and in fact $\r_\star$ is the unique solution. On the other hand, $\X_\star$ is only one solution among a set of solutions for (\ref{eq:a1}), but is the unique solution for problem (\ref{eq:a2}). Now consider the following problem
\begin{equation}\label{eq:a4}
\begin{aligned}
\minimize_{\r\in\C^N,\X\in\H^N_+}~~ & \left\|\b - \real{\F_M\Itilde\r}\right\|^2\\
\st~~~~ & r_k = \tr{\T_k\X}, ~~k = 0, 1, ..., N-1,\\
& X_{00} \geq X_{\star00}.
\end{aligned}
\end{equation}
It is easy to see that $(\r_\star,\X_\star)$ is its unique solution. From Lagrange duality, the above problem has the same solution as
\begin{equation}\label{eq:a5}
\begin{aligned}
\minimize_{\r\in\C^N,\X\in\H^N_+}~~ & \left\|\b - \real{\F_M\Itilde\r}\right\|^2 + \lambda(X_{\star00} - X_{00})\\
\st~~~~ & r_k = \tr{\T_k\X}, ~~k = 0, 1, ..., N-1,\\
\end{aligned}
\end{equation}
where $\lambda>0$ is the optimal Lagrangian multiplier with respect to the equality constraint $X_{00} = X_{\star00}$ in (\ref{eq:a4}).

Finally, by dropping the constant $\lambda X_{\star00}$ in the cost of (\ref{eq:a5}) and eliminating the variable $\r$ as we described in Appendix~\ref{appx:sdp}, we end up with the formulation (\ref{eq:a3}), or (\ref{prob:phaselift1}), leading to our conclusion that $\X_\star$ is the unique solution to (\ref{eq:a3}).
\end{proof}

Proposition~\ref{ppst:phaselift-sf} allows to solve 1D Fourier phase retrieval via solving (\ref{prob:phaselift1}), to which all methods provided by TFOCS can still be applied, and by tuning $\lambda$ we obtain an optimal solution that is also rank one.

\section{Derivation of Algorithm~(\ref{alg:admm})}\label{appx:admm}
We first review the alternating direction method of multipliers (ADMM), which is the tool used to derive algorithm (\ref{alg:admm}).

Consider the following optimization problem:
\begin{align*}
\minimize_{\x,\z}~~& f(\x) + g(\z), \\
\st~~ & \A\x + \B\z = \bm{c}.
\end{align*}
ADMM solves this problem using the following updates
\begin{align*}
\x &\leftarrow \arg\min_{\x} f(\x) + \rho\|\A\x + \B\z - \bm{c} + \bm{u}\|^2, \\
\z &\leftarrow \arg\min_{\z} g(\z) + \rho\|\A\x + \B\z - \bm{c} + \bm{u}\|^2, \\
\bm{u} &\leftarrow \bm{u} + \A\x + \B\z - \bm{c},
\end{align*}
\eldar{where $\rho>0$ is the step size.}
ADMM converges to an optimal solution as long as the problem is closed, convex, and proper. The flexible two-block structure of the algorithm often leads to very efficient and/or parallel algorithms that work well in practice.

We now apply ADMM to problem (\ref{prob:admm}), which leads to the following steps:
\begin{align*}
\r &\leftarrow \arg\min_{\r} \left\|\b \!-\! \real{\F_M\Itilde\r}\right\|^2 \!+
						\rho \left\|\real{\F_L\Itilde\r} \!-\! \z \!+\! \bm{u}\right\|^2 \\
\z &\leftarrow \arg\min_{\z} {\cal I}_+(\z) +
						\rho \left\|\real{\F_L\Itilde\r} - \z + \bm{u}\right\|^2 \\
\bm{u} &\leftarrow \bm{u} + \real{\F_L\Itilde\r} - \z
\end{align*}
The explicit update for $\z$ is very straight forward---it is simply a projection of the point
$\Re\{\F_L\Itilde\r\} + \bm{u}$
onto the non-negative orthant, leading to the update of $\z$ as shown in (\ref{alg:admm}).

We next derive the update for $\r$ in detail. First note that $\r$ is in general complex, so that the derivative needs to be taken with care.
To this end we treat $\r$ and $\r^*$ as independent variables and take derivatives with respect to them separately. This approach is referred to as the \emph{Wirtinger derivative}~\cite{wirtinger1927formalen}. Focusing on the first term of the $\r$ sub-problem, we can rewrite it as
\begin{equation}\label{eq:r-update}
\left\| \b - \frac{1}{2}\F_M\Itilde\r - \frac{1}{2}\F_M^*\Itilde\r^* \right\|^2.
\end{equation}
Taking the derivative with respect to $\r$ while treating $\r^*$ as an independent variable, we obtain
\begin{equation}\label{eq:grad}
\half\Itilde\F_M^H\F_M^{}\Itilde\r + \half\Itilde\F_M^H\F_M^*\Itilde\r^* - \Itilde\F_M^H\b.
\end{equation}

Recall that $\F_M$ represents the first $N$ columns of the DFT matrix, which are orthogonal to each other, so that $\F_M^H\F_M^{} = M\eye$. On the other hand, the columns of $\F_M^*$, except for the first one, come from the last $N-1$ columns of the same DFT matrix; if we assume $M\geq 2N$, then we have that all columns of $\F_M^*$ are orthogonal to columns of $\F_M$, except for the first one, which is equal to $\bm{1}$. Therefore, $\F_M^H\F_M^* = M\bm{E}_{00}$, where $\bm{E}_{00}$ has only one entry in the upper-left corner that is equal to one, and zeros elsewhere. Since $\Itilde=\diag{[~1~2~2~...~2~]}$, (\ref{eq:grad}) simplifies to
\[
\frac{M}{2}\Itilde\Itilde\r + \half\Itilde\bm{E}_{00}\Itilde\r^* - \Itilde\F_M^H\b
= M\Itilde\r - \Itilde\F_M^H\b,
\]
where we used the fact that $r_0$ is real. Similar expressions apply to the second term of the $\r$ sub-problem. By setting the gradient of the $\r$ sub-problem equal to zero, we have that
\[
M\Itilde\r - \Itilde\F_M^H\b + \rho\left( L\Itilde\r - \Itilde\F_L^H(\z-\bm{u}) \right) = 0,
\]
which leads to the update for $\r$ given in (\ref{alg:admm}). The gradient with respect to $\r^*$ can be shown to be exactly the conjugate of the gradient of $\r$. Therefore, equating it to zero gives the same result.

Finally, if we restrict $\r$ to be real, then the term (\ref{eq:r-update}) can be written as
\[
\left\| \b - \frac{1}{2}\F_M\Itilde\r - \frac{1}{2}\F_M^*\Itilde\r \right\|^2.
\]
Its gradient with respect to $\r$ is then
\begin{align*}
  & \half\Itilde\F_M^H\F_M^{}\Itilde\r + \half\Itilde\F_M^H\F_M^*\Itilde\r - \Itilde\F_M^H\b \\
  & + \half\Itilde\F_M^T\F_M^{}\Itilde\r + \half\Itilde\F_M^T\F_M^*\Itilde\r - \Itilde\F_M^T\b \\
=~& \frac{M}{2}\Itilde\Itilde\r + \frac{M}{2}\Itilde\bm{E}_{00}\Itilde\r - \Itilde\F_M^H\b \\
  & + \frac{M}{2}\Itilde\bm{E}_{00}\Itilde\r + \frac{M}{2}\Itilde\Itilde\r - \Itilde\F_M^T\b \\
=~& 2M\Itilde\r - 2\real{\Itilde\F_M^H\b}.
\end{align*}
Therefore, it is easy to see that the update of $\r$ that is real is
\[
\r = \real{\frac{1}{M+\rho L}\left(\F_M^H\b + \rho\F_L^H(\z-\bm{u})\right)},
\]
which is simply a projection of the original complex update onto the real domain.

\section*{Acknowledgment}
The authors would like to thank Prof. Moti Segev and Maor Mutzafi from the Technion for pointing out the similarities between our proposed measuring technique and holography, {\color{blue}and Pontus Giselsson for providing important references on the convergence rate of ADMM}.

\bibliographystyle{unsrt}
\bibliography{refs}

\end{document}